\documentclass{article}
\usepackage[utf8]{inputenc}

\usepackage{amscd,amsthm,amsfonts,amsopn,amssymb,mathtools}
\usepackage{enumitem}
\usepackage{fullpage}
\usepackage{thm-restate}
\usepackage{comment}
\usepackage[dvipsnames]{xcolor}
\usepackage{tikz}
\usepackage{hyperref}

\usepackage{bbm}

\newtheorem{thm}{Theorem}[section]

\newtheorem{lemma}[thm]{Lemma}
\newtheorem{prop}[thm]{Proposition}
\newtheorem{cor}[thm]{Corollary}

\newcommand{\Z}{\mathbb{Z}}
\newcommand{\eq}{\operatorname{eq}}

\newcommand{\dsp}{\displaystyle}

\newcommand{\rb}{\operatorname{rb}}

\title{Rainbow Numbers of $\Z_{n}$ for $a_1x_1+a_2x_2+a_3x_3 =b$}
\author{Katie Ansaldi\thanks{Wabash College, ansaldik@wabash.edu}, Houssein El Turkey\thanks{University of New Haven, helturkey@newhaven.edu}, Jessica Hamm\thanks{Winthrop University, 	hammj@winthrop.edu}, Anisah Nu'Man\thanks{Spelman College, anisah.numan@spelman.edu},\\ Nathan Warnberg\thanks{University of Wisconsin-La Crosse, nwarnberg@uwlax.edu}, Michael Young\thanks{Iowa State University, myoung@iastate.edu}}
\date{\today}

\usepackage{lineno}

\begin{document}
	\maketitle
	
	\begin{abstract}
		An exact $r$-coloring of a set $S$ is a surjective function $c:S\to [r]$.  The rainbow number of a set $S$ for equation $eq$ is the smallest integer $r$ such that every exact $r$-coloring of $S$ contains a rainbow solution to $eq$. In this paper, the rainbow number of $\Z_p$, for $p$ prime and the equation $a_1x_1 + a_2x_2 + a_3x_3 = b$ is determined. The rainbow number of  $\Z_{n}$, for a natural number $n$, is determined under certain conditions. 
	\end{abstract}

	\section{Introduction}
	
	Let $c$ be a coloring of set $S$.  A subset $X\subseteq S$ is rainbow if each element of $X$ is colored distinctly.  For example, color $[n] = \{1,2,\dots,n\}$ and consider solutions to the equation $x_1 + x_2 = x_3$.  If each element of a solution $\{a,b,a+b\}\subseteq [n]$ is colored distinctly, that solution is rainbow. One of the first papers to investigate rainbow arithmetic progressions is \cite{J}, where Jungi\'c et al. showed that colorings with each color used equally yield rainbow arithmetic progressions.  In \cite{J}, only $3$-term arithmetic progressions are considered which are also solutions to $x_1+x_2 = 2x_3$. In \cite{AF}, Axenovich and Fon-Der-Flaass showed that no $5$-colorings avoid rainbow $3$-term arithmetic progressions. A few articles investigated the anti-van der Waerden number, which is the fewest number of colors need to guarantee a rainbow arithmetic progression. For example, Butler et al. established, in \cite{DMS}, bounds for the anti-van der Waerden number when coloring $[n]$ and some exact values when coloring $\Z_n$. Later, Berikkyzy, Schulte, and Young determined, in \cite{BSY}, the anti-van der Waerden number for $[n]$ in the case of $3$-term arithmetic progressions. 
	
	Some of this work was generalized to graphs and abelian groups. Montejano and Serra investigated, in \cite{MS}, rainbow-free colorings of abelian groups when considering arithmetic progressions. Similarly, rainbow arithmetic progressions in finite abelian groups were studied by co-author Young, in \cite{finabgroup}, where the anti-van der Waerden numbers were connected to the order of the group. When arithmetic progressions were extended to graphs,  Rehm, Schulte, and Warnberg showed, in \cite{RSW}, the anti-van der Waerden numbers on graph products is either $3$ or $4$.
	
	Generalizing the equation $x_1 + x_2 = 2x_3$, Bevilacqua et al., in \cite{BKKTTY}, considered $x_1 + x_2 = kx_3$ on $\Z_n$. The rainbow number of $\Z_n$ was determined for these equations when $k = 1$ or $k = p$ where $p$ is prime.  These results served as motivation for this paper where the equation $a_1x_1 + a_2x_2+a_3x_3 = b$ will be considered over $\Z_{p}$ and $\Z_{n}$ with $p$ prime. From now on $a_1x_1 + a_2x_2+a_3x_3 = b$ will be denoted by $\eq(a_1,a_2,a_3,b)$.  This paper establishes the rainbow number, also known as the anti-van der Waerden number, of $\Z_{n}$ for $\eq(a_1,a_2,a_3,b)$ for some equations.  One important result that will be used is Huicochea and Montejano's characterization, in \cite{RFC}, of all rainbow-free exact $3$-colorings of $\Z_p$ for $\eq(a_1,a_2,a_3, b)$ for all primes $p$.

	\subsection{Preliminaries}\label{sec:prelims}    
	An \emph{$r$-coloring} of a set $S$ is a function $c:S \to [r]$ and an $r$-coloring is \emph{exact} if $c$ is surjective.  Note that an exact $r$-coloring yields a partition of $S$ into $r$ disjoint color classes.   This paper will focus on the linear equation $\eq(a_1,a_2,a_3,b)$ given by \begin{equation}\label{eq1} a_1x_1+ a_2x_2 + a_3x_3 = b\end{equation}
	
	\noindent and $r$-colorings of $\mathbb{Z}_n$.  An ordered set $(s_1,s_2,s_3)$ is called a \emph{solution} to $\eq(a_1,a_2,a_3,b)$ in $\Z_n$  if $a_1s_1 + a_2s_2+a_3s_3 \equiv b\bmod n$.  Throughout the paper $=$ will be used instead of $\equiv$, and the $\bmod$ $n$ will not be used unless the context requires clarification.
	
	If $c$ is an $r$-coloring of $\Z_n$, then a \emph{rainbow solution} in $\Z_n$ to $\eq(a_1,a_2,a_3,b)$ is a solution such that $|\{c(s_1), c(s_2),c(s_3)\}| = 3$, i.e. each member of the solution has been assigned a distinct color by $c$.  A coloring $c$ of $\Z_n$ is \emph{rainbow-free} for $\eq(a_1,a_2,a_3,b)$ if there are no rainbow solutions.
	
	The rainbow number of $\Z_n$ for equation $eq = \eq(a_1,a_2,a_3,b)$, denoted $\rb(\Z_n, eq)$, is the smallest positive integer $r$ such that every exact $r$-coloring of $\Z_n$ has a rainbow solution for $eq$.  If there are no rainbow solutions to $eq$ in an exact $n$-coloring of $\Z_n$, then the convention will be that $\rb( \Z_n, eq) = n + 1$.  Since rainbow solutions to $eq$ require three colors, then $\rb(\Z_n, eq) \ge 3$, for all $n\ge2$.
	
	The following tools will be used throughout the paper.  Given a set $S\subseteq \Z_n$ and $d,t\in \Z_n$, the sets $S+t = \{s+t\,|\, s\in S\}$ and $dS = \{ds\,|\,s\in S\}$ are called the \emph{$t$-translation} and \emph{$d$-dilation} of $S$, respectively. If the multiplicative inverse of $a\in \Z_n$ exists, denote the inverse by $a^{-1}$. The set of all these invertible elements forms a group under multiplication, and it is denoted by  $\Z_n^*$. For $d\in \Z_n^*$, let $\langle d \rangle$ be the multiplicative subgroup of $\Z_n^*$ generated by $d$ and $\langle d_1,\dots,d_k\rangle$ be multiplicatively generated by the $d_i$'s. A subset $S\subseteq \Z_n$ is \emph{$\langle d\rangle$-periodic} if $S=dS$ and a set is called \emph{symmetric} if it is $\langle -1 \rangle$-periodic. For ease of reading, the related results from \cite{RFC} are referenced below.

	\begin{thm}\label{thm3-HM}(\cite[Theorem~3]{RFC}). Let $A$, $B$ and $C$ be the color classes of an exact $3$-coloring of $\Z_p$ such that $1\leq |A| \leq |B| \leq |C|$.  The coloring is rainbow-free for $\eq(1,1,-c,0)$ if and only if, under dilation, one of the following holds true: 
		\begin{enumerate}
			\item[1)] $A=\{0\}$, with both $B$ and $C$ symmetric $\langle c \rangle$-periodic subsets.
			\item[2)] $A=\{1\}$ for 
			\begin{enumerate}
				\item[a)] $c=2$, with $(B-1)$ and $(C-1)$ symmetric $\langle 2 \rangle$-periodic subsets;
				\item[b)] $c=-1$, with $(B \backslash \{-2\}) +2^{-1}$ and $(C \backslash \{-2\})+2^{-1}$ symmetric subsets.
			\end{enumerate}
			\item[3)] $|A| \geq 2$, for $c=-1$, with $A, B$ and $C$ arithmetic progressions with difference 1, such that $A=\{i\}_{i=t_1}^{t_2 -1}$, $B=\{i\}_{i=t_2}^{t_3 -1}$, and $C=\{i\}_{i=t_3}^{t_1 -1}$, where $(t_1+t_2+t_3)=1$ or $2$. 
		\end{enumerate}
	\end{thm}

	\begin{thm}\label{thm6-HM}(\cite[Theorem~6]{RFC}). Let $A$, $B$ and $C$ be the color classes of an exact $3$-coloring of $\Z_p$ such that $1\leq |A| \leq |B| \leq |C|$. The coloring is rainbow-free for $\eq(a_1,a_2,a_3,b)$, with some $a_i\neq a_j$,
		if and only if $A=\{s\}$ with $s(a_1+a_2+a_3)=b$, and both $B$ and $C$ are sets invariant under six specific transformations.
	\end{thm}

	\begin{cor}\label{cor8-HM}(\cite[Corollary~8]{RFC}).
		Every exact $3$-coloring of $\Z_p$ contains a rainbow solution of\\ $\eq(a_1,a_2,a_3,b)$, with some $a_i \neq a_j$, if and only if one of the following holds true:
		\begin{itemize}
			\item[1)] $a_1+a_2+a_3=0 \not= b$,
			\item[2)] $|\langle d_1, \dots, d_6\rangle|=p-1$,
		\end{itemize}
		
		where $d_1 = -a_3a_1^{-1}$,
		$d_2 = -a_2a_1^{-1}$,
		$d_3 = -a_1a_2^{-1}$,
		$d_4 = -a_3a_2^{-1}$,
		$d_5 = -a_1a_3^{-1}$, and $d_6 = -a_2a_3^{-1}$.
	\end{cor}

	Note that Theorem \ref{thm5-HM} is the same as the case when $b=0$ and $c=-1$ in Theorem \ref{thm3-HM}.  It is included for completion.

	\begin{thm}\label{thm5-HM}\cite[Theorem~5]{RFC}.  Let $A$, $B$ and $C$ be the color classes of an exact $3$-coloring of $\Z_p$ with $p\ge 3$ and $1\leq |A| \leq |B| \leq |C|$.  The coloring is rainbow-free for $\eq(1,1,1,b)$ if and only if one of the following holds true: 
		\begin{itemize}
			\item[1)] $A=\{s\}$ with both $(B \backslash \{b-2s\}) +(s-b)2^{-1}$ and $(C \backslash \{b-2s\})+(s-b)2^{-1}$ symmetric sets.
			\item[2)] $|A| \geq 2$, and all $A, B$ and $C$ are arithmetic progressions with the same common difference $d$, so that $d^{-1}A=\{i\}_{i=t_1}^{t_2 -1}$, $d^{-1}B=\{i\}_{i=t_2}^{t_3 -1}$, and $d^{-1}C=\{i\}_{i=t_3}^{t_1 -1}$ satisfy $t_1+t_2+t_3 \in \{1+d^{-1}b, 2+d^{-1}b\}$.
		\end{itemize}
	\end{thm}
	
	Lemma \ref{lem:J} will be used to extend results for rainbow numbers of equations where $b=0$ to equations where $b\ne 0$.
	\begin{lemma}\label{lem:J}
		For $a_1,a_2,a_3\in \Z_n$ let $a=a_1+a_2+a_3$ and suppose that $a\in \Z_n^{\ast}$. There exists a rainbow-free $k$-coloring of $\Z_n$ for $a_1x_1+a_2x_2+a_3x_3=b$ if and only if there exists a rainbow-free $k$-coloring  of $\Z_n$ for  $a_1x_1+a_2x_2+a_3x_3=0$.
	\end{lemma}
	\begin{proof} 
		Define $T: \Z_n \to \Z_n$ by $T(x)= x-ba^{-1}$. Suppose $(s_1,s_2,s_3)$ is a solution to $a_1x_1+a_2x_2+a_3x_3=b$. Applying the one-to-one transformation $T$ to $(s_1, s_2, s_3)$ gives:
		\begin{eqnarray*}
			a_1T(s_1)+a_2T(s_2)+a_3T(s_3) & = & a_1(s_1-ba^{-1}) + a_2(s_2-ba^{-1}) + a_3(s_3-ba^{-1}) \\
			& =& a_1s_1+a_2s_2+a_3s_3 + (a_1+a_2+a_3)(-ba^{-1}) \\
			&=& b + a(-ba^{-1}) \\
			&=& 0.
		\end{eqnarray*}
		Similarly, if $(T(s_1),T(s_2),T(s_3))$ is a solution to $a_1x_1+a_2x_2+a_3x_3=0$, then $(s_1,s_2,s_3)$ is a solution to $a_1x_1+a_2x_2+a_3x_3=b$. This gives a one-to-one correspondence between solutions of the two equations.
	\end{proof}
	
	This paper is organized as follows.  First, $\rb(\Z_p,\eq(a_1,a_2,a_3,b))$, is determined in Section~\ref{sec:zp}.  The main result in this section, Theorem~\ref{thm:rbzp}, states that the rainbow number of $\Z_p$ is either $3$ or $4$. In Section~\ref{sec:zn}, the rainbow number of $\Z_{n}$ is computed for a natural number $n$. The main result in this section, Theorem~\ref{thm:rbzn}, shows that for $n = p_1^{\alpha_1} p_2^{\alpha_2}\cdots p_\ell^{\alpha_\ell}$
	
	$$\rb(\Z_n,\eq(a_1,a_2,a_3,b)) =  2+ \dsp\sum_{k = 1}^\ell\left[\alpha_k(\rb(\Z_{p_k},\eq(a_1,a_2,a_3,b) - 2)\right],$$
	
	\noindent under certain conditions.  To prove this result, Section~\ref{sec:zn} includes leading lemmas and theorems such as finding the rainbow number $\rb(\Z_{2^{\alpha}},\eq(a_1,a_2,a_3,b))$ in Theorem~\ref{thm:23}, establishing the right hand side of the above equation as a lower bound in Corollary~\ref{thm:lowerboundzn-b}, and as an upper bound in Corollary~\ref{cor:upperboundzn}.

	\section{Rainbow Numbers of $\Z_p$}\label{sec:zp}
	
	This section establishes the rainbow number for the equation $\eq(a_1,a_2,a_3,b)$ over $\Z_p$ where $p$ is a prime.  Under certain conditions, Lemma \ref{lem-Anisah} establishes that if two elements in a solution are the same, then all three are the same.  This fact was mentioned in \cite{RFC} without proof and has been included for completion.
	\begin{lemma}\label{lem-Anisah}
		If $a_1s_1+a_2s_2+a_3s_3=0$ over $\Z_p$ with $|\{s_1,s_2,s_3\}| <3$, $a_1+a_2+a_3=0$ and $a_1a_2a_3\in \Z_p^{\ast}$, then $s_1=s_2 =s_3$.  
	\end{lemma}
	\begin{proof} If $s_1 = s_2 = s_3$ the proof is complete.  Without loss of generality, assume $s_1 = s_2$.
		Observe $a_1+a_2+a_3=0$ implies $a_3=-a_1-a_2$, and therefore $a_1s_1+a_2s_1+a_3s_3=(a_1+a_2)(s_1-s_3)=0.$
		
		Since $\Z_p$ has no zero divisors this gives $a_1+a_2=0$ or $s_1-s_3=0$.  Note $a_1+a_2=0$ along with $a_1+a_2+a_3=0$ gives $a_3=0$ which contradicts $a_1a_2a_3\not=0$. Therefore, $s_1-s_3=0$ and so $s_1=s_3$ and $|\{s_1,s_2,s_3\}| = 1$.
	\end{proof}
	
	If $p=2$, by convention $\rb(\Z_2,\eq)=3$. The case when $p=3$ is handled next.

	\begin{prop}\label{z3}
		For all $a_1,a_2,a_3,b\in \Z_3$,
		\[
		\rb(\Z_3,\eq(a_1,a_2,a_3,b) )=\begin{cases}
		3 \quad \text{if $b=0$ and $a_i=a_j$, ~ for some $i \not=j$}\\
		\quad ~ \text{or $b \not=0$ and $a_i\not=a_j,$ ~ for some $i \ne j$,}\\
		4 \quad \text{otherwise}.
		\end{cases}
		\]
	\end{prop}
	\begin{proof}
		
		Note there is only one way (up to isomorphism) to color $\Z_3$ with three distinct colors. Suppose $\eq$ has a rainbow solution and, without loss of generality, assume a solution is $(1,2,0)$. For $\eq(a_1, a_2, a_3, 0)$, $a_1+2a_2=0$ implies $a_1=a_2$. It then follows that a rainbow solution will exist if and only if $a_i=a_j$ for some $i\ne j$, giving $\rb(\Z_3, \eq(a_1,a_2,a_3,0 )=3$. If the $a_i$'s are all distinct, by standard convention, $\rb(\Z_3,\eq(a_1,a_2,a_3,0))=4$. 
		
		Now consider $\eq(a_1, a_2, a_3, b)$ for $b\ne 0$. The solution $(1,2,0)$ gives $a_1-a_2=b $. Since $b\ne 0$, then $a_1 \ne a_2$. It then follows that $\rb(\Z_3,\eq(a_1,a_2,a_3,b))=3$ if $a_i\ne a_j$ for some $i\ne j$. Otherwise,  $\rb(\Z_3,\eq(a_1,a_2,a_3,b))=4$.
	\end{proof}
	
	Next, the case $p\geq 5$ will be discussed. Theorem \ref{thm:rbzp} shows that the rainbow number of $\eq(a_1,a_2,a_3,b)$ is either $3$ or $4$ depending on the different variations of $a_1, a_2, a_3$ and $b$.  The following theorem also uses notation established in Corollary \ref{cor8-HM}.
	
	\begin{thm}\label{thm:rbzp}
		Let $a_1,a_2,a_3,b\in \Z_p$ with some $a_i\ne a_j$ and $a_1a_2a_3\in \Z_p^{\ast}$ for $p\ge5$,
		then
		$$\rb(\Z_p,\eq(a_1,a_2,a_3,b)) = 
		\left
		\{\begin{array}{ll}
		3 & \text{if $|\langle d_1,d_2,\dots, d_6\rangle| = p-1$}\\
		& \text{or $a_1+a_2+a_3=0\neq b$,}\\
		4 &\text{otherwise.}
		\end{array}  \right.$$
	\end{thm}
	
	\begin{proof} The proof follows by case analysis.  First, define $eq = \eq(a_1,a_2,a_3,b)$. \\
		\textbf{Case 1}: $|\langle d_1,d_2,\dots , d_6\rangle |=p-1$ or $a_1+a_2+a_3 = 0 \neq b$
		
		\noindent The conditions in this case are the conditions of Corollary \ref{cor8-HM}, thus $\rb(\Z_p,eq)\leq 3$ and $\rb(\Z_p, eq)=3$.
		
		\noindent\textbf{Case 2}: $|\langle d_1,d_2,\dots , d_6\rangle|<p-1$ and $a_1+a_2+a_3 \neq 0$\\
		By Corollary \ref{cor8-HM}, there exists a rainbow-free $3$-coloring which implies $\rb(\Z_p,eq)\geq 4$. Since  $a_1+a_2+a_3 \neq 0$ there is a unique $s\in\Z_p$ such that $s(a_1+a_2+a_3)=b$. Suppose there is a 4-coloring of $\Z_p$ with color classes $A$, $B$, $C$, and $D$ such that $s\in A$. Create a $3$-coloring with color classes $A\cup B$, $C$, and $D$.  By construction, $s$ is not in a color class by itself. Theorem \ref{thm6-HM} now guarantees there is a rainbow solution in this $3$-coloring which corresponds to a rainbow solution in the $4$-coloring.  Thus, $\rb(\Z_p,eq)\leq 4$ and hence, $\rb(\Z_p,eq)=4$.
		
		\noindent {\bf Case 3}: $|\langle d_1,d_2,\dots , d_6\rangle|<p-1$, $a_1+a_2+a_3 = 0$, and $b=0$\\
		Since $|\langle d_1,d_2,\dots , d_6\rangle|<p-1$ and $b=0$, both conditions in Corollary \ref{cor8-HM} fail; hence, $\rb(\Z_p,eq)\geq 4$.  Note that in this case, every $s\in\Z_p$ satisfies $s(a_1+a_2+a_3)=b$.
		To show that $\rb(\Z_p,eq)\leq 4$, consider a $4$-coloring of  $\Z_p$.\\
		\underline{Case 3.1:} At most two color classes have size one. \\
		Combine the two smallest color classes to make a $3$-coloring that has no color classes of size one. By Theorem \ref{thm6-HM}, this $3$-coloring contains a rainbow solution. Thus, the original $4$-coloring contains a rainbow solution, which implies $\rb(\Z_p,eq)\leq 4$. \\
		Note, if there are at least three color classes of size one, then the argument used in Case $3.1$ does not hold.  Essentially, combining the two smallest color classes will give a $3$-coloring that has a color class with one element.\\
		\underline{Case 3.2}: At least three color classes have size one.\\
		Let $A=\{s_1\}$ and $B=\{s_2\}$ be two of the three color classes of size one.  Let $s_3 = a_3^{-1}(-a_1s_1-a_2s_2)$, then $(s_1,s_2,s_3)$ is a solution. Since $s_1\ne s_2$, by Lemma~\ref{lem-Anisah}, then $s_1$, $s_2$, $s_3$ are distinct.  Therefore, $(s_1, s_2, s_3)$ is a rainbow solution. Thus, the $4$-coloring contains a rainbow, which implies $\rb(\Z_p,eq)\leq 4$. 
	\end{proof}
	
	Note that Theorem \ref{thm:rbzp} considered equations where $a_i\neq a_j$ for some $i\neq j$.  For the remainder of this section it is assumed that $a_1=a_2=a_3$.  To handle equations of this type, Theorem \ref{thm3-HM} and Lemma \ref{apl} are essential.
	
	\begin{lemma}\label{apl} Suppose sets $A$, $B$, $C$, and $D$ partition $\Z_p$. The sets $A \cup B$, $A \cup C$, $A \cup D$, $B$, $C$, and $D$ cannot all be arithmetic progressions with common difference $d\neq 0$.
	\end{lemma}
	
	\begin{proof} For the sake of contradiction, suppose $A \cup B$, $A \cup C$, $A \cup D$, $B$, $C$, and $D$ are all arithmetic progressions with common difference $d$.  Define $B = \{\beta, \beta + d, \dots, \beta + kd\}$.  Since $B$ and $A\cup B$ are both arithmetic progressions with the same common difference, then $A$ contains $\beta - d$ or $\beta + (k+1)d$.  Similarly, this applies to $C$ and $A\cup C$ and applies to $D$ and $A\cup D$.  However, this implies that $B$, $C$, and $D$ are not pairwise disjoint, a contradiction.
	\end{proof}
	
	\begin{thm}\label{equalupperbound2} If $a\in \Z_p^*$, $b\in \Z_p$,  and $p\ge5$, then $\rb(\Z_p,\eq(a,a,a,b)) = 4$.
	\end{thm}
	
	\begin{proof} 
		Since $p\ge 5$, then $3a \in\Z_p^*$. By Lemma \ref{lem:J}, it is enough to consider $ax_1+ax_2+ax_3=0$. Furthermore, because $a\in\Z_p^*$, the triple $(s_1,s_2,s_3)$ is a solution to $ax_1+ax_2+ax_3 = 0$ if and only if it is a solution to $x_1+x_2+x_3=0$. Thus, without loss of generality, the rest of the argument only considers $x_1+x_2+x_3 = 0$. The exact rainbow-free $3$-coloring of $\Z_p$ with color classes $\{0\}, \{1,p-1\}, \{2,3,\dots,p-2\}$ establishes that $4\le \rb(\Z_p,\eq(1,1,1,0))$.  Suppose there is an exact $4$-coloring of $\Z_p$ with color classes $A$, $B$, $C$, and $D$ such that $|A| \le |B| \le |C| \le |D|$.  It will be shown that a rainbow solution exists in the aforementioned exact $4$-coloring.\\ 
		\textbf{Case 1:} At most one color class has size one.\\
		Consider the exact $3$-colorings with color classes: $A\cup B$, $C$, $D$; $B$, $A\cup C$, $D$; and $B$, $C$, $A\cup D$.  If each of them is rainbow-free, then, by Theorem \ref{thm5-HM}, each of $A \cup B$, $A \cup C$, $A \cup D$, $B$, $C$, and $D$ are arithmetic progressions with the same common difference $d$.  This contradicts Lemma \ref{apl} so one of the exact $3$-colorings must have a rainbow solution.\\
		\textbf{Case 2:} Exactly two color classes have size one. \\
		Let $A = \{s\}$ and $B = \{\beta\}$.  If $\beta \neq -2s$, then $\{s, \beta, -s-\beta\}$ is a rainbow solution.  Thus, without loss of generality, assume $\beta = -2s$.  Note this also means $s \neq 0$.  Consider the exact $3$-coloring with color classes $A\cup B$, $C$, $D$.  If this coloring is rainbow-free, then, by Theorem \ref{thm5-HM}.$ii$, $A\cup B$, $C$ and $D$ must be arithmetic progressions with common difference $d$.  Further, $d^{-1}(A\cup B)$, $d^{-1}C$ and $d^{-1}D$ are sets of consecutive integers and is a rainbow-free exact $3$-coloring.  Now consider the exact $3$-coloring with color classes $d^{-1}A$, $d^{-1}(B\cup C)$, $d^{-1}D$.  Theorem \ref{thm5-HM}.$i$ implies that $d^{-1}(B\cup C)\backslash\{d^{-1}(-2s)\} + s2^{-1} = d^{-1}C + s2^{-1}$ and $d^{-1}D + s2^{-1}$ are symmetric.  However, the color classes must also be consecutive which would imply $\beta = -s$, which is a contradiction.\\
		\textbf{Case 3:}  At least three color classes have size one.\\
		Without loss of generality, dilate the coloring so that $A=\{1\}$, $B = \{\beta\}$, and $C=\{\gamma\}$.  Note that if the exact $3$-colorings with color classes $A$, $C$, $B\cup D$ and $A$, $B$, $C\cup D$ are rainbow-free, then they must be of the form described in Theorem \ref{thm3-HM} part $2.b$.  This means $B\backslash\{-2\} +2^{-1}\in \{\emptyset,\{0\}\}$ and $C  \backslash\{-2\} +2^{-1}\in\{\emptyset,\{0\}\}$.  So, without loss of generality, $\beta = -2$ and $\gamma = -(2^{-1})$.  Notice that $(-2,-(2^{-1}), 2+2^{-1})$ is a rainbow solution because $-2 = -(2^{-1})$, $2+2^{-1} = -2$ or $2+2^{-1} = -(2^{-1})$ imply $p\in \{2,3\}$.\\
		
		In all cases, an exact $3$-coloring constructed from the original exact $4$-coloring has a rainbow solution. Thus the original exact $4$-coloring has a rainbow solution.\end{proof}
	
	
	\section{Rainbow Numbers of $\Z_n$ for $a_1x_1+a_2x_2 + a_3x_3 =b$}\label{sec:zn}
	
	In this section the rainbow number for $\Z_n$ will be established under certain conditions on the coefficients. Since $2$ is a special case, the rainbow number for $\Z_{2^\alpha}$ will be considered first. Then the lower and upper bounds are established for general $n$.  
	
	Let $A$ and $B$ be sets and $m,n\in \Z$.  $(A,m,eq)$ is \emph{solomorphic} to $(B, n, eq')$ when there exists a function $\phi:A\to B$ such that $\{s_1,s_2,s_3\} \subset A$ is a solution to $eq \bmod m$ if and only if $\{\phi(s_1),\phi(s_2),\phi(s_3)\} \subset B$ is a solution to $eq' \bmod n$.  Note that solomorphic sets have the same rainbow number.
	
	\begin{thm}\label{thm:23}
		If $a_1a_2a_3 \in \Z_{2}^{\ast}$, then $$\rb(\Z_{2^\alpha},\eq(a_1,a_2,a_3,b))=\alpha+2.$$ 
	\end{thm}
	
	\begin{proof} The proof follows by induction on $\alpha$. The base case $\alpha=1$ holds by convention. Note that since $a_1a_2a_3 \in \Z_{2}^{\ast}$, then $a_i\equiv 1\bmod 2$ for all $i$ and, by Lemma~\ref{lem:J}, it can be assumed that $b=0$. 
		Let $0\leq \alpha \in \Z$ and assume the statement holds for $\alpha\geq 1$. The following cases show the statement is true for $\alpha +1$.  .   Let $c$ be an exact $\alpha+2$ coloring of $\Z_{2^\alpha}$ and define $R_i=\{x\in \Z_{2^\alpha} \mid x\equiv i \bmod 2\}$ and $P_i=\{c(x) \mid x \in R_i\}$. 
		
		If at least $\alpha+1$ colors appear in $P_0$, then, by the inductive hypothesis, $\Z_{2^\alpha}$ contains a rainbow solution to $eq=\eq(a_1,a_2,a_3,0)$ because $(\Z_{2^{\alpha-1}}, 2^{\alpha-1}, eq)$ is solomorphic to $(R_0,2^\alpha,eq)$. If at most $\alpha$ colors appear in $P_0$, then there exist two colors, red and blue, that appear in $P_1$. Let $s_1$ and $s_2$ be two elements in $R_1$ such that $c(\{s_1,s_2\}) = \{red,blue\}$. Since $s_1\equiv s_2\equiv 1 \bmod 2$, there exists $s_3\in \Z_{2^\alpha}$ such that $a_1(s_1+s_2)+a_1s_3\equiv 0\bmod 2$, which implies $s_3\in R_0$.  This means $\{s_1,s_2,s_3\}$ is a rainbow solution to $\eq(a_1,a_2,a_3,0)$ since $c(s_3) \notin \{red,blue\}$.  Thus $\rb(Z_{2^\alpha},\eq(a_1,a_2,a_3,0)) \le \alpha+2$.  To obtain a lower bound, color $P_0$ with a rainbow-free coloring of $\Z_{2^{\alpha-1}}$ that has $\alpha$ colors and color $P_1$ with the $(\alpha+1)^{st}$ color.  This coloring has no rainbow solutions since every solution has exactly $1$ or $3$ elements from $R_0$.
	\end{proof}

	Theorem~\ref{thm:lowerbound1} and Corollary~\ref{cor:lowerboundzn} establish the lower bound for the rainbow number of $\Z_n$.
	
	\begin{thm}\label{thm:lowerbound1}
		Let $2\le t\in\Z$. If $ a_1a_2a_3 \in \Z_{p}^{\ast}$, $p$ prime, then $$\rb(\mathbb{Z}_p, \eq(a_1,a_2,a_3,0)) + \rb(\mathbb{Z}_t,\eq(a_1,a_2,a_3,0)) -2\le \rb (\mathbb{Z}_{pt}, \eq(a_1,a_2,a_3,0)). $$
	\end{thm}
	
	\begin{proof}   Since $\rb(\Z_p,eq)\le 4$, every rainbow-free coloring of $\Z_p$ for $eq$ uses 
		at most three colors. Define $r_p =\rb(\Z_p, eq)-1$ and $r_t=\rb(\Z_t, eq) -1$. Note there exists an exact rainbow-free $r_p$-coloring of $\Z_p$ for $eq$ where $0$ is the only element in its color class. If $r_p=2$ or $p=3$, the $r_p$-coloring is obvious. If $r_p=3$, $p\ge 5$, and $a_i \neq a_j$, for some $i \neq j$, such a coloring exists by Theorem \ref{thm6-HM}.   Lastly, if $a_1=a_2=a_3$ the coloring is described in Theorem \ref{equalupperbound2}.
		
		Let $c_p$ be an exact, rainbow-free $r_p$-coloring of $\Z_p$ for $eq$ such that $0$ is in a color class of its own and $c_t$ be an exact, rainbow-free $r_t$-coloring of $\Z_t$ for $eq$.  Define an exact $(r_p+r_t-1)$-coloring of $\Z_{pt}$ by
		
		$$c(x)=\left\{\begin{array}{ll}
		0 & \text{if $x = 0$},\\
		c_p(x\bmod p) & \text{if $x\neq 0 \bmod p$},\\
		(r_p-1)+c_t\left(\frac{x}{p} \bmod t \right) & \text{if $x= 0\bmod p$ and $x\neq 0$}.
		\end{array}
		\right.
		$$
		
		Let $(s_1, s_2, s_3)$ be a solution in $\Z_{pt}$ to $eq$.  Since $a_1a_2a_3\in\Z_{p}^*$, $p$ cannot divide exactly two of $s_1$, $s_2$, and $s_3$, so either $p$ divides each of $s_1$, $s_2$, and $s_3$ or $p$ divides at most one of $s_1$, $s_2$, and $s_3$.
		
		If $p$ divides each of $s_1$, $s_2$, and $s_3$, then $(s_1,s_2,s_3)$ is not a rainbow solution under the coloring $c$ since it is not a rainbow solution under $c_t$.  If $p$ divides at most one of $s_1$, $s_2$, and $s_3$, then $(s_1,s_2,s_3)$ is not a rainbow solution under the coloring $c$ since it is not a rainbow solution under $c_p$ and $0$ is the unique element in its color class under $c_p$.  Therefore, $c$ is a rainbow-free coloring of $\Z_{pt}$.  This implies 
		$$\rb(\Z_{pt},eq) \ge (r_p+r_t-1)+1 = \rb(\mathbb{Z}_p, eq) + \rb(\mathbb{Z}_t,eq) -2.$$
	\end{proof}

	\begin{cor}\label{cor:lowerboundzn} If $n = p_1^{\alpha_1}p_2^{\alpha_2}p_3^{\alpha_3} \ldots p_\ell^{\alpha_\ell}$, $p_k$ prime for $1\le k \le \ell$,  and $a_1a_2a_3\in \Z_{n}^\ast$, then
		
		$$2 + \sum_{i=1}^\ell \left[ \alpha_i (\rb (\Z_{p_i}, \eq(a_1,a_2,a_3,0) -2) \right] \le rb(\mathbb{Z}_n, \eq(a_1,a_2,a_3,0)). $$
	\end{cor}

	\begin{proof}
		Define $eq = \eq(a_1,a_2,a_3,0)$ and note that if $n$ is prime there is nothing to show. Suppose that the claim holds for all $k < n$. Consider $k = \frac{n}{p_j} = p_1^{\alpha_1} p_2^{\alpha_2} \ldots p_j^{\alpha_j - 1} \ldots p_\ell^{\alpha_\ell}$. By the inductive hypothesis, 
		
		$$\rb (\mathbb{Z}_{k}, eq) \geq 2 + (\alpha_j-1)(\rb (\Z_{p_j}, eq)-2) + \sum_{i=2}^\ell  \alpha_i \left[\rb (\Z_{p_i}, eq) -2 \right].$$
		
		Applying Theorem \ref{thm:lowerbound1} gives \begin{eqnarray*}
			\rb (\mathbb{Z}_{\frac{n}{p_j} \cdot p_j}, eq) &\geq& \rb(\mathbb{Z}_{p_j}, eq) + \rb(\mathbb{Z}_k,eq) -2\\
			&\geq& \rb(\mathbb{Z}_{p_j}, eq) +  2 + (\alpha_j-1)(\rb (\Z_{p_j}, eq)-2) + \sum_{i=2}^\ell \alpha_i \left[\rb (\Z_{p_i}, eq) -2 \right]-2\\
			&= &2+ \sum_{i=1}^\ell \alpha_i\left[  \rb (\Z_{p_i}, eq) -2 \right],
		\end{eqnarray*}
		as desired.
	\end{proof}
	
	Corollary~\ref{thm:lowerboundzn-b} generalizes Corollary~\ref{cor:lowerboundzn} to $\eq(a_1,a_2,a_3,b)$ using Lemma~\ref{lem:J}. 
	\begin{cor}\label{thm:lowerboundzn-b}
		If $n = p_1^{\alpha_1}p_2^{\alpha_2}p_3^{\alpha_3} \ldots p_\ell^{\alpha_\ell}$, $p_k$ prime for $1\le k \le \ell$,  and  $a_1+a_2+a_3, a_1a_2a_3\in \Z_{n}^\ast$, then
		
		$$2 + \sum_{i=1}^\ell \left[ \alpha_i (\rb (\Z_{p_i}, \eq(a_1,a_2,a_3,b) -2) \right]\le \rb(\mathbb{Z}_n, \eq(a_1,a_2,a_3,b)).$$
	\end{cor}

	The upper bound will now be established. Suppose $c$ is a coloring of $\Z_{ut}$. The remainder of this section uses residue classes $R_i=\{z\in \Z_{ut}\mid z\equiv i \bmod u\}$ and color palettes $P_i=\{ c(z) \mid z\in R_i\}$ that were mentioned in the proof of Theorem \ref{thm:23}.
	
	\begin{lemma}
		\label{obs:pods}
		Let $3\le t,u\in \Z$, $a_3\in \Z_u^{\ast}$, and  $(s_1,s_2,s_3)$ and $(s_1',s_2',s_3')$ be solutions in $\Z_{ut}$ to $\eq(a_1,a_2,a_3,b)$. If $s_1'\in R_{{s_1}}$ and $s_2'\in R_{{s_2}}$, then $s_3'\in R_{{s_3}}$. 
	\end{lemma}
	\begin{proof}
		Since $a_1s_1' + a_2s_2' + a_3s_3' = b \bmod ut$ implies $a_1s_1' + a_2s_2' + a_3s_3' = b \bmod u$, solving for $s_3'$ over $\Z_u$ gives
		$$s_3'=a_3^{-1}(b-(a_1s_1+a_2s_2)) = a_3^{-1}(a_3s_3) \bmod u.$$
		Hence, $s_3' \in R_{{s_3}}$.  
	\end{proof}
	
	A similar argument to the one used in Lemma \ref{obs:pods} can be used for $a_1,a_2\in \Z_t^\ast$ and solving for $s_1'$ and $s_2'$ instead.
	
	\begin{lemma}\label{colortrans} If $k,n \in \mathbb{Z}$ such that $3\le n$, then $$\rb(\mathbb{Z}_{n},\eq(a_1,a_2,a_3,b)) = \rb(\mathbb{Z}_{n}, \eq(a_1,a_2,a_3,b+ (a_1+a_2+a_3)k)).$$
	\end{lemma}
	
	\begin{proof} Let $eq = \eq(a_1,a_2,a_3,b)$, $a = a_1+a_2+a_3$, $eq' = \eq(a_1,a_2,a_3,b+ak)$ and $c$ be an exact $r$-coloring of $\Z_{n}$ for $eq$. If $(s_1,s_2,s_3)$ is a solution in $\Z_{n}$ to $eq$, then $a_1s_1+a_2s_2+a_3s_3 + (a_1+a_2+a_3)k = b+ak$ and $(s_1+k,s_2+k,s_3+k)$ is a solution in $\Z_{n}$ to $eq'$.  Define $c_{k}:\Z_{n} \to [r]$ by $c_{k}(x) = c(x+k \bmod n)$.  Thus, $(s_1,s_2,s_3)$ is a rainbow solution to $eq$ with respect to $c$ if and only if $(s_1+k, s_2+k, s_3+k)$ is a rainbow solution to $eq'$ with respect to $c_{k}$.  Since $c_k$ is a translation of the coloring $c$, $\rb(\Z_{n},eq) = \rb(\Z_{n},eq')$.\end{proof}

	\begin{lemma}\label{lem:chat}
		Let $ 2\le t\in\Z$ and $c$ be a rainbow-free coloring of $\Z_{ut}$ for $\eq(a_1,a_2,a_3,b)$ and $a_1a_2a_3 \in \Z_{ut}^\ast$ that does not use color $yellow$. If there exists $j\in \Z_t$ such that for all $i\in\Z_t$, $|P_{i}\backslash P_{j}|\le 1$, then the coloring of $\Z_t$ given by
		$$\hat{c}(i)=\left\{ \begin{array}{ll}
		{\it yellow} & P_{i}\subseteq P_{j}, \\
		P_{i}\setminus P_{j} & \textrm{otherwise},
		\end{array}
		\right.$$
		is well-defined and rainbow-free.
	\end{lemma}

	\begin{proof} Since $|P_{i}\backslash P_{j}|\le 1$, $\hat{c}$ is well-defined. Let $eq=\eq(a_1,a_2,a_3,b)$ and
		assume that $(s_1,s_2,s_3)$ is a rainbow solution of $eq$ in $\Z_t$ with respect to $\hat{c}$. Since $(s_1,s_2,s_3)$ is a rainbow solution, without loss of generality, $\hat{c}(s_1)=red$ and $\hat{c}(s_2)=blue$. Thus, there exist $\alpha\in R_{s_1}, \, \delta\in R_{s_2}$ such that $c(\alpha)=red$ and $c(\delta)=blue$. Therefore, $(\alpha, \delta, \gamma)$  is a solution to $eq$ in $\Z_{ut}$ for some $\gamma\in R_{{s_3}}$. Note that $\hat{c}(s_3)$ is not {\it red} or {\it blue}. However, $P_{{s_3}}\setminus \{\hat{c}(s_3)\}\subseteq P_{j}$. Therefore $c(\gamma)$ is not {\it red} or {\it blue} so $(\alpha, \delta, \gamma)$ is a rainbow solution
		to $eq$ in $\Z_{ut}$ with respect to $c$, a contradiction.
	\end{proof}
	
	\begin{lemma}\label{p0max}
		If $c$ is a rainbow-free coloring of $\Z_{ut}$ for $\eq(a_1,a_2,a_3,b)$, $a_1a_2a_3\in \Z_{ut}^\ast$ and $|P_0| \ge |P_i|$ for $0 \le i \le u-1$, then $|P_i\backslash P_0| \le 1$.
	\end{lemma}
	\begin{proof}
		Assume $|P_i \backslash P_0| \ge 2$ for some $1 \le i \le u-1$ and let $red, blue\in P_i \backslash P_0$. Let $j \in \Z_{ut}$ such that $a_1 i + a_2 0 + a_3 j = b$.  Suppose there is an $\alpha\in R_j$ such that $c(\alpha)\notin P_0$.  Choose $\beta\in R_i$ such that $c(\beta) \in\{red,blue\}\backslash \{c(\alpha)\}$.  Now there exists $\gamma\in R_0$ such that $\{\beta,\gamma,\alpha\}$ is a rainbow solution to $\eq(a_1,a_2,a_3,b)$, a contradiction.  Therefore, $P_j\subseteq P_0$.  A similar argument gives that $P_0 \subseteq P_j$, so $P_0 = P_j$.
		
		
		
		
		Since $|P_0|$ is maximum there must exist two colors, both in $P_0$ and $P_j$, that are not in $P_i$. Let $yellow,green \in P_0 \backslash P_i$.  Choosing a \emph{yellow} element in $R_0$ and a \emph{green} element in $R_j$ and solving for the appropriate element in $R_i$ will give a rainbow solution, which is a contradiction. Therefore, $|P_i \backslash P_0| \le 1$ for all $0 \le i\le u-1$.
	\end{proof}

	Using an inductive argument with the following Lemma \ref{zeropalettechatv2}, similar to the argument made in Corollary \ref{cor:lowerboundzn}, and Theorem \ref{thm:23} gives Corollary \ref{cor:upperboundzn}.
	
	\begin{lemma}\label{zeropalettechatv2}
		If $ 2\le t\in\Z$, $3 \le p$ prime, $a_1a_2a_3 \in \Z_{pt}^\ast$, then $$\rb(\Z_{pt},\eq(a_1,a_2,a_3,b)) \le \rb(\Z_p,\eq(a_1,a_2,a_3,b_2)) + \rb(\Z_t,\eq(a_1,a_2,a_3,b_1)) - 2,$$ for some $b_1,b_2 \in \mathbb{Z}$.
	\end{lemma}

	\begin{proof}
		Let $eq = \eq(a_1,a_2,a_3,b)$ and $c$ be a rainbow-free exact $(\rb(\Z_{pt},\eq(a_1,a_2,a_3,b))-1)$-coloring of $\Z_{pt}$. Create the coloring $c_k$ from Lemma \ref{colortrans} to get $|P_0| \ge |P_i|$ for $1\le i \le p-1$, where $P_i$ are defined with respect to coloring $c_k$. This implies that $(\Z_{pt},pt,eq)$ is solomorphic to $(\Z_{pt},pt,eq_1)$ with $eq_1 = \eq(a_1,a_2,a_3,b_1)$ for some $b_1 \in \Z_{pt}$.
		
		Since $c_k$ is rainbow-free and $|P_0| \ge |P_i|$ for all $i$, Lemma \ref{lem:chat} and Lemma \ref{p0max} give a well-defined coloring $\hat{c}$ using $P_0$.  If $\hat{c}$ has a rainbow solution, then $c_k$ has a rainbow solution, so $\hat{c}$ must be rainbow-free.  However, since $\hat{c}$ is coloring $\Z_t$, $\hat{c}$ uses at most $\rb(\Z_t,eq_1) - 1$ colors which contributes at most $\rb(\Z_t,eq_1) -2$ colors to $c$ because, without loss of generality, $yellow$ is not a color from $c$.  Furthermore, $(R_0,pt,eq_1)$ is solomorphic to $(\Z_p,p,eq_2)$ so $|P_0| \le \rb(\Z_p,eq_2) - 1$, where $eq_2 = \eq(a_1,a_2,a_3,b_2)$ for some $b_2 \in \Z_{p}$.  In order for $\hat{c}$ to be rainbow-free, $c_k$ must use at most $\rb(\Z_p,eq_2) + \rb(\Z_t,eq_1) - 3$ colors. This implies $\rb(\Z_{pt},\eq(a_1,a_2,a_3,b)) - 1 \le \rb(\Z_p,eq_2) + \rb(\Z_t,eq_1) - 3$.
	\end{proof}
	
	\begin{cor}\label{cor:eqk}
		If $n = p_1p_2\cdots p_\ell$, $3\leq p_k$ prime for $1\le k \le \ell$, $a_1a_2a_3\in \Z_n^\ast$, and $eq = \eq(a_1,a_2,a_3,b)$, then
		$$\rb(\Z_n,eq)\le 2+ \dsp\sum_{k = 1}^\ell\left[\rb(\Z_{p_k},eq_k) - 2)\right],$$
		where $eq_k=\eq(a_1,a_2,a_3,b_k)$ for some $b_k\in \Z$.
	\end{cor}
	
	\begin{cor}\label{cor:eq0}
		Let $n = p_1p_2\cdots p_\ell$, $p_k$ prime for $1\le k \le \ell$, $a_1a_2a_3\in \Z_n^\ast$, where $a_1+a_2+a_3\in \Z_3^{\ast}$ if $3 \mid n$.   Let $eq = \eq(a_1,a_2,a_3,0)$, then
		$$\rb(\Z_n,eq)\le 2+ \dsp\sum_{k = 1}^\ell\left[\rb(\Z_{p_k},eq) - 2)\right].$$
	\end{cor}
	\begin{proof}  
		By Theorems~\ref{thm:rbzp}, \ref{equalupperbound2}, and \ref{thm:23}, if $p\neq 3$, then $\rb(\Z_p, \eq(a_1,a_2,a_3,b))\leq \rb(\Z_p, \eq(a_1, a_2, a_3, 0))$ for all $b$. If $a_1+a_2+a_3\in \Z_3^{\ast}$, Proposition~\ref{z3} gives $\rb(\Z_3, \eq(a_1,a_2,a_3,b))\leq \rb(\Z_3, \eq(a_1, a_2, a_3, 0))$ for all $b$. The result follows by  Corollary~\ref{cor:eqk}. 
	\end{proof}
	Note that  the assumption $a_1+a_2+a_3\in \Z_{3}^\ast$ is necessary when $3\mid n$. For example,  $\rb(\Z_{3},\eq(1,1,1,0)) = 3$ and $\rb(\Z_{9},\eq(1,1,1,0)) = 5$.  In particular, $\Z_9$ has the rainbow-free coloring $c:\Z_9\rightarrow [4]$ given by $c(2)=2$, $c(5)=3$, $c(8)=4$, and $c(x)=1$ else.
	
	Corollary \ref{cor:eq0} and Lemma~\ref{lem:J} combine to give Corollary \ref{cor:upperboundzn}.
	
	\begin{cor}\label{cor:upperboundzn}
		If $n = p_1^{\alpha_1}p_2^{\alpha_2}\cdots p_\ell^{\alpha_\ell}$, $p_k$ prime for $1\le k \le \ell$, $a_1+a_2+a_3, a_1a_2a_3 \in \Z_n^\ast$ and $eq = \eq(a_1,a_2,a_3,b)$, then
		
		$$\rb(\Z_n,eq)\le    2+\dsp\sum_{k = 1}^\ell\left[\alpha_k(\rb(\Z_{p_k},eq) - 2)\right]. $$
		
	\end{cor}

	Finally, Corollaries \ref{cor:lowerboundzn}, \ref{thm:lowerboundzn-b} and \ref{cor:upperboundzn} combine to give Theorem \ref{thm:rbzn}.
	
	\begin{thm}\label{thm:rbzn}Let $n = p_1^{\alpha_1} p_2^{\alpha_2}\cdots p_\ell^{\alpha_\ell}$, with $p_k$ prime for $1\le k \le \ell$, and $ a_1a_2a_3\in \Z_n^\ast$.  If one of the following holds:
		\begin{enumerate}
			\item[1)] $b\neq0$ and $a_1+a_2+a_3\in \Z_n^\ast$,
			\item[2)] $b=0$ and $3 \nmid n$, or
			\item[3)] $b=0$, $3\mid n$, and $a_1+a_2+a_3 \in \Z_3^\ast$,
		\end{enumerate}
		
		then
		
		$$\rb(\Z_n,\eq(a_1,a_2,a_3,b))= 2+ \dsp\sum_{k = 1}^\ell\left[\alpha_k(\rb(\Z_{p_k},\eq(a_1,a_2,a_3,b) - 2)\right].$$
	\end{thm}
	
	\section*{Acknowledgements}
	We greatly appreciate the feedback from the referee as it has improved the results and structure of the paper. This work initiated at the 2018 Research Experiences for Undergraduate Faculty Workshop (REUF) hosted at the American Institute of Mathematics (AIM) in San Jose, CA. REUF is a program of the AIM and the Institute for Computational and Experimental Mathematics (ICERM), made possible by the support from the National Science Foundation (NSF) through DMS 1239280. We also thank AIM for supporting our research retreats including funding a week-long meeting in Summer 2019. The last author is also supported by the NSF Award \# 1719841. The second author is supported by a University Research Scholar fellowship from his institution.


\begin{thebibliography}{99}
		
		\bibitem{AF}
		M. Axenovich and D. Fon-Der-Flaass, On rainbow arithmetic progressions, \emph{Electronic Journal of Combinatorics} {\bf 11} (2004), 7 pp.
		
		\bibitem{BSY}
		Z. Berikkyzy, A. Schulte, and M. Young, Anti-van der Waerden numbers of 3-term arithmetic progressions, \emph{Electronic Journal of Combinatorics} {\bf 24(2)} (2017), 9 pp.
		
		\bibitem{BKKTTY} E. Bevilacqua, A. King, J. Kritschgau, M. Tait, S. Tebon and M. Young, Rainbow numbers for $x_1 + x_2 = kx_3$ in $\mathbb{Z}_n$, \url{https://arxiv.org/abs/1809.04576}.
		
		\bibitem{DMS}
		S. Butler, C. Erickson, L. Hogben, K. Hogenson, L. Kramer, R.L. Kramer, J. Lin, R.R. Martin, D. Stolee, N. Warnberg and M. Young, Rainbow arithmetic progressions, \emph{Journal of Combinatorics} {\bf 7(4)} (2016), 595--626.
		
		\bibitem{RFC} M. Huicochea and A. Montejano, The structure of rainbow-free colorings for linear equations on three variables in $\mathbb{Z}_p$, \emph{Integers}  Volume 15A, A8, (2015).
		
		\bibitem{J}
		V. Jungi\'c, J. Licht (Fox), M. Mahdian, J. Ne\u{s}etril, and R. Radoi\u{c}i\'c, Rainbow arithmetic progressions and anti-Ramsey results, \emph{Combinatorics Probability and Computing} {\bf 12} (2003), no 5-6, 599--620.
		
		\bibitem{LM} B. Llano and A. Montejano, Rainbow-free colorings for $x+y = cz$ in $\Z_p$, \emph{Discrete Mathematics} {\bf 312} (2012), 2566--2573.  
		
		\bibitem{MS} A. Montejano and O. Serra, Rainbow-free $3$-colorings in abelian groups, \emph{Electronic Journal of Combinatorics} {\bf 34} (2012), 5 pp. 
		
		\bibitem{RSW} H. Rehm, A. Schulte and N. Warnberg, Anti-van der Waerden numbers on graph products, \emph{Australasian Journal of Combinatorics} {\bf 73(3)} (2019), 486--500.
		
		
		\bibitem{finabgroup} M. Young, Rainbow arithmetic progressions in finite abelian groups, \emph{Journal of Combinatorics} {\bf 9(4)} (2018), 619--629.
	\end{thebibliography}
\end{document}